\documentclass[7pt,a4paper,twocolumn]{article}
\usepackage[utf8]{inputenc}
\usepackage[T1]{fontenc}
\usepackage{mathptmx}
\usepackage[pdftex]{graphicx}
\usepackage[pdftex,linkcolor=black,pdfborder={0 0 0}]{hyperref}
\usepackage{calc}
\usepackage{enumitem}

\usepackage[a4paper, lmargin=0.08\paperwidth, rmargin=0.08\paperwidth, tmargin=0.08\paperheight, bmargin=0.08\paperheight]{geometry}

\usepackage[all]{nowidow}
\usepackage[protrusion=true,expansion=true]{microtype}

\usepackage{amsmath}
\usepackage{amsthm}
\newtheorem{theorem}{Theorem}
\usepackage{amsfonts}
\usepackage{complexity}
\usepackage{mathtools}

\graphicspath{{./figures/}}

\usepackage{titlesec}
\titlespacing*{\section}{0pt}{1pt}{1pt}
\titlespacing*{\subsection}{0pt}{1pt}{1pt}
\titlespacing*{\paragraph}{0pt}{1pt}{1pt}

% ...
\usepackage[capitalize]{cleveref}
\usepackage{float}
\usepackage{subcaption}
\usepackage{aliascnt}
\newaliascnt{problem}{theorem}
\newtheorem{problem}[problem]{Problem}
\aliascntresetthe{problem}

\crefname{problem}{problem}{problems}
\Crefname{problem}{Problem}{Problems}

\DeclareMathOperator{\cartesian}{\Box}

\author{Samuel Schneider, Torsten Ueckerdt}
\title{Number of Edges in 3-Connected Graphs with Cyclic Neighborhoods}
\makeatletter

\hypersetup{
    pdfsubject = {3-Connected Graphs with Cyclic Neighborhoods},
    pdftitle = {\@title},
    pdfauthor = {\@author}
}

\usepackage{fancyhdr}
\begin{document}    
    An \emph{independent cut} of a graph $G$ is a vertex set $S\subset V(G)$ such that $S$ is an independent set and $G - S$ is disconnected.
    Chen and Yu~\cite{edge_bound_independent_cut} show that every $n$-vertex graph with at most $2n-4$ edges has an independent cut.
    This bound is tight as the graph consisting of $n-2$ triangles sharing one edge has no independent cut and $2n-3$ edges. 
    
    Chernyshev, Rauch and Rautenbach~\cite{conjecture_cycle_in_neighborhood} introduce \emph{forest cuts}, i.\,e., vertex separators that induce a forest.
    They conjecture that, similar to the result by Chen and Yu, every $n$-vertex graph with less than $3n-6$ edges has a forest cut.\footnote{According to Chernyshev, Rauch and Rautenbach~\cite{conjecture_cycle_in_neighborhood}, this has been independently conjectured by Atsushi Kaneko at the 7th C5 Graph Theory Workshop (Kurort Rathen) in 2003.} 
    As an intermediate goal they state the following problem:

    \begin{problem}\label{problem}
        How many edges must an $n$-vertex $3$-connected graph have such that the neighborhood of every vertex contains a cycle?
    \end{problem}

    Li, Tang and Zhan~\cite{long_proof} resolve \Cref{problem}:
    Every such graph has at least $\frac{15}{8}n$ edges, while there are such graphs with exactly $\frac{15}{8}n$ edges.
    We give a much shorter proof for this.

	\begin{theorem} \label{thm:upper_bound}
        Let $G$ be a $3$-connected $n$-vertex graph such that for every vertex $v\in V(G)$ the graph induced by its neighborhood $G[N(v)]$ contains a cycle.
        Then $|E(G)| \geq \frac{15}{8}n$.
    \end{theorem}
    
    \begin{proof}
        Since $G$ is $3$-connected, every vertex has degree at least~$3$.
        Consider the partition of $V(G)$ into the sets $V_3 \coloneq \{v \in V(G) \mid \deg(v) = 3\}$ and $V_{\geq 4} \coloneq \{ v\in V(G) \mid \deg(v) \geq 4\}$.
        Note that the neighborhood of each $v \in V_3$ induces a triangle since $G[N(v)]$ contains a cycle.
        
        First, we show that $V_3$ is an independent set in $G$.
        Assume this is not the case.
        Then, there is a vertex $v \in V_3$ that has a neighbor $u \in V_3$.
        As both $N(u)$ and $N(v)$ induce triangles,
        it follows that $N(u) \cap N(v)$ is a separator of size~$2$ (see \Cref{fig:V_3_independent}) --- 
        A contradiction to $G$ being $3$-connected.

        \begin{figure}[H]
            \begin{minipage}{0.45\linewidth}
                \centering
                \includegraphics[page=1]{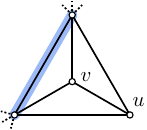}
                \caption{If two vertices $u,v\in V_3$ are adjacent, then $N(u) \cap N(v)$ is a separator of size~$2$ (highlighted in blue).}
                \label{fig:V_3_independent}
            \end{minipage}
            \hfill
            \begin{minipage}{0.45\linewidth}
                \centering
                \includegraphics[page=2]{figures/V_3_independent.pdf}
                \caption{Vertices $u,v \in V_3$ with $N(u)=N(v)$ and $u' \in N(u)$. A cycle in $N(u') - u$ highlighted in blue.}
                \label{fig:distinct_neighborhoods}
            \end{minipage}
        \end{figure}

        Now let us assume that there are two vertices $u,v \in V_3$ with $N(u) = N(v)$.
        Then the graph $G-u$ is still $3$-connected.
        Furthermore, every neighbor $u'$ of $u$ still has a cycle in its neighborhood in $G-u$, namely a triangle formed by $v$ and the other two neighbors of $u$ (see \Cref{fig:distinct_neighborhoods}).
        If $G$ has at least $\frac{15}{8}n$ edges we are clearly done. 
        Otherwise, $G-u$ has less edges than $G$ (compared to their number of vertices).
        Thus we may assume that for any two vertices $u,v \in V_3$ we have $N(u) \neq N(v)$.

        Next, we show that for every vertex $v \in V_{\geq 4}$ we have that $|N(v) - V_3| \geq 3$.
        As $\deg(v) \geq 4$, this clearly holds if $v$ has at most one neighbor in $V_3$.
        So assume $u,u' \in N(v) \cap V_3$ with $u \neq u'$.
        As $V_3$ is an independent set, all neighbors of $u$ and $u'$ are in $V_{\geq 4}$.
        Furthermore, each of $N(u)$ and $N(u')$ induces a triangle in $G$.
        As $v \in N(u) \cap N(u')$, it follows that $N(u) \cup N(u') \subseteq N(v) \cup \{v\}$.
        Now since $N(u) \neq N(u')$ (as argued earlier), this gives the desired $|N(v) - V_3| \geq 3$.

        Thus, for every $v \in V_{\geq 4}$ we have $\deg(v) \geq 3 + |N(v) \cap V_3|$.
        As every vertex in $V_3$ has three neighbors in $V_{\geq 4}$, we get:
        \begin{equation}
            2|E(G)| \geq \sum_v \deg(v)\geq 3n + 3|V_3| \label{eq:1}
        \end{equation}
        On the other hand, we can compute:
        \begin{equation}
            2|E(G)| = \sum_v \deg(v) \geq 4n - |V_3| \label{eq:2}
        \end{equation}
        The sum of \eqref{eq:1} and three times \eqref{eq:2} gives the desired:
        \[
            8|E(G)| \geq 3n+3|V_3| + 3(4n-|V_3|) = 15n \qedhere
        \]
    \end{proof}

    \begin{theorem}\label{thm:lower_bound}
        There are $3$-connected $n$-vertex graphs $G$ with $\frac{15}{8}n$ edges such that for every vertex $v\in V(G)$ the graph induced by its neighborhood $G[N(v)]$ contains a cycle.
    \end{theorem}

    \begin{proof}
        We start with a $3$-connected $3$-regular graph~$H$, e.\,g., $H = C_t \cartesian K_2$.
        We construct $G$ by replacing each vertex $v$ of $H$ by a copy of $K_4$ such that the three edges incident to $v$ in $H$ are incident to different vertices of the $K_4$ in $G$,
        see \Cref{fig:construction}.

        \begin{figure}[H]
            \centering
            \includegraphics{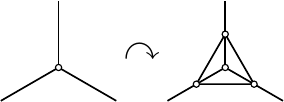}
            \caption{The substitution applied to every vertex of $H$.}
            \label{fig:construction}
        \end{figure}

        It is easy to see that $G$ is $3$-connected since $H$ is $3$-connected.
        Moreover, as each vertex $v \in V(G)$ is contained in a copy of $K_4$, the subgraph $G[N(v)]$ contains a triangle.

        Doing the counting, each vertex of $H$ corresponds in $G$ to a copy of $K_4$ with four vertices and six edges.
        Thus, we have $|V(G)| = 4\cdot |V(H)|$ and $|E(G)| = |E(H)| + 6\cdot |V(H)|$.
        Since $H$ is $3$-regular, we have $|E(G)| = \frac{15}{2}\cdot|V(H)| = \frac{15}{8}|V(G)|$.     
    \end{proof}
	
	\bibliographystyle{plainurl}
	\bibliography{bibliography.bib}

@article{conjecture_cycle_in_neighborhood,
title = {Forest cuts in sparse graphs},
journal = {Discrete Mathematics},
volume = {348},
number = {11},
pages = {114594},
year = {2025},
issn = {0012-365X},
doi = {https://doi.org/10.1016/j.disc.2025.114594},
author = {V. Chernyshev and J. Rauch and D. Rautenbach}
}

@article{edge_bound_independent_cut,
title = {A note on fragile graphs},
journal = {Discrete Mathematics},
volume = {249},
number = {1},
pages = {41-43},
year = {2002},
note = {Combinatorics, Graph Theory, and Computing},
issn = {0012-365X},
doi = {https://doi.org/10.1016/S0012-365X(01)00226-6},
author = {G. Chen and X. Yu},
}

@misc{long_proof,
      title={The minimum size of a $3$-connected locally nonforesty graph}, 
      author={C. Li and Y. Tang and X. Zhan},
      year={2024},
      eprint={2410.23702},
      archivePrefix={arXiv},
      primaryClass={math.CO},
}
\end{document}